\documentclass[11pt,reqno]{amsart}

\usepackage{amsfonts, amsmath, amssymb, amsthm}
\usepackage{anysize}
\usepackage{mathtools}
\marginsize{2cm}{2cm}{1.5cm}{1.5cm}

\usepackage[utf8]{inputenc}
\usepackage[english]{babel}
\usepackage{cmap}
\usepackage{enumerate}

\usepackage{xcolor}
\usepackage{hyperref}
\hypersetup{
	colorlinks   = true, 
	urlcolor     = blue, 
	linkcolor    = blue, 
	citecolor    = red 
}
\newcommand{\Href}[2]{\hyperref[#2]{#1~\ref{#2}}}
\usepackage{anysize}
\marginsize{2cm}{2cm}{1.5cm}{1.5cm}

\usepackage{cancel}

\theoremstyle{theorem}
\newtheorem{thm}{Theorem}[section]
\newtheorem{prp}{Proposition}[section]
\newtheorem{lem}{Lemma}[section]
\newtheorem{cor}{Corollary}[section]
\newtheorem{conj}{Conjecture}[section]

\theoremstyle{definition}
\newtheorem{dfn}{Definition}[section]
\newtheorem{rem}[dfn]{Remark}

\newtheorem{prb}{Problem}

\newcommand{\norm}[1]{\left\|#1\right\|}

\newcommand{\EE}{{\mathbb E}}

\DeclareMathOperator{\tr}{\mathrm{trace}}

\providecommand{\parenth}[1]{\left(#1\right)}%
\providecommand{\braces}[1]{\left\{#1\right\}}%
\def\R{{\mathbb R}}%

\def\phi{\varphi}
\def\epsilon{\varepsilon}

\newcommand{\id}{\mathrm{Id}}

\newcommand{\e}{\varepsilon}

\newcommand{\di}{\,\mathrm{d}}

\DeclarePairedDelimiter{\floor}{\lfloor}{\rfloor}

\newcommand{\Mone}{M_1}
\newcommand{\Mtwo}{M_2}
\newcommand{\pot}[2]{F_{#1}\!\left(#2\right)}

\title{Greedy sparsifications of sums of positive semidefinite matrices}
\author{Grigory Ivanov}
\address{Grigory Ivanov: Pontifícia Universidade Cat'olica do Rio de Janeiro \
Departamento de Matematica, \
Rua Marquês de São Vicente, 225 \
Edif{'i}cio Cardeal Leme, sala 862, \
22451-900 G{'a}vea, Rio de Janeiro, Brazil}
\email{grimivanov@gmail.com}
\thanks{The author is supported by Projeto Paz and Coordenacao de Aperfeicoamento de Pessoal de Nivel Superior - Brasil (CAPES) - 23038.015548/2016-06.}
\date{}

\begin{document}

\begin{abstract}
We prove a deterministic analogue of Rudelson's sampling theorem for sums of positive semidefinite matrices. Let \(A_1,\dots,A_m\) be positive semidefinite \(d\times d\) matrices, and let \(\lambda_1,\dots,\lambda_m\ge 0\) satisfy
\[
\sum_{i \in [m]} \lambda_i=1,
\qquad
\sum_{i \in [m]} \lambda_i A_i=\id_d,
\qquad
\norm{A_i}\le M
\quad\text{for all }i\in \braces{1, \dots, m}.
\]
We show that there exists a deterministic sequence of indices \(i_1,i_2,\dots\in \braces{1, \dots, m}\) such that for every integer \(k\ge 1\),
\[
\norm{
\frac1k\sum_{r \in [k]} A_{i_r}-\id_d
}
\le
\begin{cases}
2\,\dfrac{M\ln(2d)}{k}, & \text{if } k\le M\ln(2d),\\[2ex]
3\,\sqrt{\dfrac{M\ln(2d)}{k}}, & \text{if } k> M\ln(2d).
\end{cases}
\]
In particular, if \(0<\epsilon\le 1\) and \(N\ge 9M\ln(2d)\epsilon^{-2}\), then one can choose \(i_1,\dots,i_N\in \braces{1, \dots, m} \) so that
\[
\norm{
\frac1N\sum_{r \in [N]} A_{i_r}-\id_d
}
\le \epsilon.
\]
The construction is greedy and deterministic, and it yields the correct logarithmic order simultaneously for every prefix of the sequence.
\end{abstract}

\maketitle

\section{Introduction}

Sparsification problems play a central role in modern mathematics and its applications. Broadly speaking, one starts with a structured representation of an object as a large sum and asks whether it can be replaced by a much shorter sum while preserving the relevant analytic, spectral, or geometric information. This point of view appears in spectral graph theory, numerical linear algebra, randomized algorithms, approximation theory, and asymptotic geometric analysis; see, for instance, \cite{tropp2015introduction, vershynin2018high, artstein2021asymptotic}.

In this paper we study a basic matrix sparsification problem. Let
\(
A_1,\dots,A_m 
\)
be positive semidefinite \(d\times d\) matrices, and let
\[
\lambda_1,\dots,\lambda_m\ge 0,
\qquad
\sum_{i\in[m]} \lambda_i =1,
\qquad
\sum_{i\in[m]} \lambda_i A_i=\id_d,
\]
where, and throughout the paper, $[m]$ denotes the set $\{1, \dots, m\}$ for a positive integer $m$.
The goal is to approximate the identity operator by a short sum of the matrices \(A_i\).

There are two natural versions of this problem.

In the \emph{weighted} version, one is allowed to choose arbitrary non-negative coefficients and seeks a sparse vector
\(
s=(s_1,\dots,s_m)\in\R_+^m
\)
such that
\[
\norm{\sum_{i\in[m]} s_i A_i - \id_d} \leq \e.
\]
Here and throughout, \(\norm{\cdot}\) denotes the operator norm.
In the \emph{equal-weight} version, one asks for an approximation by an average of the form
\[
\frac1N\sum_{r\in[N]} A_{i_r},
\qquad i_1,\dots,i_N\in[m].
\]
These two problems are closely related, but they are not equivalent. In particular, the equal-weight setting is substantially more restrictive.

The weighted PSD sparsification problem is now fairly well understood. In the rank-one case, the fundamental theorem of Batson, Spielman, and Srivastava \cite{BSS14} yields optimal deterministic weighted sparsifiers of linear size and, in particular, linear-size spectral sparsifiers of graphs. This theory was extended to arbitrary-rank positive semidefinite matrices by de Carli Silva, Harvey, and Sato \cite{de2016sparse}; see also \cite{reis2020linear} for a probabilistic version. Thus, in the weighted setting, one essentially has linear-size sparsification in full generality.

The equal-weight problem is more subtle. A classical theorem of Rudelson \cite{rudelson1999random} shows that in the rank-one case one can obtain an \(\epsilon\)-approximation by taking
\[
N\le C\frac{d\ln d}{\epsilon^2}
\]
terms. Rudelson's method is probabilistic: it estimates the expected deviation of the empirical average and controls it via a non-commutative Khintchine-type inequality \cite{Lu86, LuP91}. Oliveira \cite{Oli2010} recast this argument in a more general setting: the same probabilistic scheme yields bounds of the form \(N\le C\frac{M \ln d}{\e^2}\) for arbitrary-rank positive semidefinite matrices as well, where \(M\) is the maximum of the norms of the original matrices.

However, the unrestricted-rank equal-weight problem behaves quite differently from its weighted counterpart. In \cite[Theorem 1.2]{IVANOV2020108684} it was shown that, for arbitrary-rank positive semidefinite matrices, the logarithmic factor is in general unavoidable. Thus the linear-size phenomenon cannot be extended to the general equal-weight model. On the other hand, in the rank-one case the situation is much better: deep results originating from the solution of the Kadison--Singer problem \cite{marcus2015interlacing} and, in particular, the work of Friedland and Youssef \cite{FY17}, show that linear-size equal-weight approximations are possible there. This leaves a striking gap between the rank-one and unrestricted-rank settings.

The main result of the present paper is a deterministic equal-weight sparsification theorem in the unrestricted-rank PSD setting. More precisely, we construct a single greedy sequence of indices such that every prefix already has the correct order of approximation. Thus the theorem simultaneously produces an average with exactly \(k\) terms for every prescribed \(k\ge 1\).
We use $\succeq$ to denote the standard L\"owner order on symmetric matrices.

\begin{thm}\label{thm:all-steps_sparsification}
Let \(A_1,\dots,A_m\succeq 0\) be \(d\times d\) matrices satisfying
\begin{equation}
\label{eq:original_sum_matrices}
\sum_{i \in [m]} \lambda_i A_i=\id_d,
\qquad
\sum_{i \in [m]} \lambda_i=1,
\qquad \lambda_i \ge 0 \ \text{ and } \
\norm{A_i}\le M
\quad\text{for all }i\in[m].
\end{equation}
Then there exists a deterministic sequence of indices
\(
i_1,i_2,\dots \in [m]
\)
such that for every integer \(k\ge 1\) one has
\[
\norm{
\frac1k\sum_{r \in [k]} A_{i_r}-\id_d
}
\le
\begin{cases}
2\,\dfrac{M\ln(2d)}{k}, & \text{if } k\le M\ln(2d),\\[2ex]
3\,\sqrt{\dfrac{M\ln(2d)}{k}}, & \text{if } k> M\ln(2d).
\end{cases}
\]
\end{thm}

In particular, the theorem gives a deterministic greedy analogue of the logarithmic Rudelson--Oliveira bound in the general PSD setting. At each step one chooses the next matrix greedily, and after any number of steps the current empirical average is already controlled. 

The usual \(\epsilon\)-approximation statement is an immediate corollary.

\begin{cor}\label{cor:epsilon}
Under the assumptions of \Href{Theorem}{thm:all-steps_sparsification}, let \(0<\epsilon\le 1\). If
\[
N\ge 9\,\frac{M\ln(2d)}{\epsilon^2},
\]
then there exists a sequence of indices
\[
i_1,\dots,i_N\in[m]
\]
such that
\[
\norm{
\frac{1}{N}\sum_{r \in [N]} A_{i_r}-\id_d
}
\le \epsilon.
\]
\end{cor}

The proof is organized in two stages. We first establish a fixed-\(N\) version by running the greedy procedure with a single value of the exponential parameter \(\delta\), and then upgrade it to the all-step statement by allowing \(\delta\) to decrease with the step number. The key input in both arguments is a deterministic one-step averaging principle driven by a symmetric exponential potential. Roughly speaking, this principle asserts that, under suitable first- and second-moment assumptions, one can choose the next summand greedily so that the potential grows at most multiplicatively at each step. In the present setting, the relevant input is not only the first-moment identity
\[
\sum_{i\in[m]} \lambda_i X_i=0,
\qquad
X_i:=A_i-\id_d,
\]
but also the quadratic bound
\[
\sum_{i\in[m]} \lambda_i X_i^2\preceq M\id_d.
\]
This second-order input is exactly what is needed for the original PSD sparsification problem. At the same time, the abstract averaging theorem proved below is formulated for self-adjoint centered matrices and no longer requires positivity of the original matrices. This auxiliary result may therefore be of independent interest.

Returning to the gap between the rank-one and unrestricted-rank problems, it is natural to ask what happens for fixed rank \(r\ge 2\).

\begin{conj}\label{conj:fixed-rank}
For every integer \(r\ge 2\) and every \(\epsilon>0\), there exists a constant \(C_{r}\) such that the following holds. Let
\(
P_1,\dots,P_m
\)
be orthogonal projections of rank \(r\) in \(\R^d\), and assume that
\[
\sum_{i\in[m]} \lambda_i P_i=\frac{r}{d}\id_d,
\qquad
\sum_{i\in[m]} \lambda_i=1, \quad
 \lambda_i \ge 0 \
\quad\text{for all }i\in[m].
\]
Then there exist indices \(i_1,\dots,i_N\in[m]\) with
\[
N\le C_{r} \frac{d}{\e^2}
\]
such that
\[
\norm{
\frac{d}{rN}\sum_{j\in[N]} P_{i_j}-\id_d
}
\le \epsilon.
\]
\end{conj}

We also mention a related question suggested by the discretization viewpoint of Temlyakov and his collaborators. In the classical theory of Marcinkiewicz-type discretization, equal-weight formulas are among the strongest and most restrictive forms of discretization, and continuous-support versions are often substantially harder than their finite-dimensional weighted analogues \cite{dai2019integral,kashin2022sampling}.  We do not address this problem here, but the analogy seems worth emphasizing.

\medskip

The coarse regime
\[
\left\|
\frac1N\sum_{r\in[N]} A_{i_r}-\id_d
\right\|
\le
2\,\frac{M\ln(2d)}{N}
\qquad\text{for }N\le M\ln(2d)
\]
also seems interesting in its own right. Although weaker than the \(\epsilon\)-approximation regime, it is well suited to geometric applications in which one seeks a reasonably well-conditioned average of a prescribed small number of operators. This is closely related in spirit to quantitative Helly- and Carath\'eodory-type phenomena; see, for example, \cite{naszodi2016proof, ivanov2024steinitz}.

We conclude with a related geometric problem, which may be viewed as a spectral analogue of the celebrated Dvoretzky--Rogers lemma \cite{dvoretzky1950absolute}.

\begin{prb}
Let
\[
\sum_{i \in [m]} c_i\,u_i\otimes u_i=\id_d
\]
be a John's decomposition of the identity, where \(u_1, \dots, u_m\) are unit vectors and \(c_1, \dots, c_m\) are positive reals. Does there always exist a choice of \(d\) contact points
\(
u_{i_1},\dots,u_{i_d}
\)
such that
\[
\norm{\id_d-\frac{1}{d} 
\sum_{r \in [d]} u_{i_r}\otimes u_{i_r}}
\le 1-\frac{1}{d^2}?
\]
\end{prb}

The natural scale of the problem would be \(1-c/d\), but even the weaker bound above already seems to be an interesting open question. For related, more geometric questions see, for example,
\cite[Conjecture 1.5]{almendra2022quantitative} and the suggested functionals in \cite{ivanov2022quantitative}.

\medskip

The paper is organized as follows. In \Href{Section}{sec:potential} we introduce the symmetric exponential potential and formulate the abstract one-step averaging theorem. We also explain why this potential is natural in the present problem.
Then, in \Href{Section}{sec:proof-main} we first prove a fixed-\(N\) version of the sparsification theorem and then upgrade it to the stronger all-step statement, namely \Href{Theorem}{thm:all-steps_sparsification}.
Finally, in \Href{Section}{sec:ave_thm} we prove the averaging theorem.

\section{The potential and the averaging principle}
\label{sec:potential}

Let \(X_1,\dots,X_m\) be self-adjoint \(d\times d\) matrices, and let
\[
\lambda_1,\dots,\lambda_m\ge 0,
\qquad
\sum_{i \in [m]} \lambda_i=1.
\]
Assume that
\begin{equation}\label{eq:mean-zero}
\sum_{i \in [m]} \lambda_i X_i=0,
\end{equation}
\begin{equation}\label{eq:norm-bound}
\norm{X_i}\le \Mone
\qquad\text{for all }i\in[m],
\end{equation}
and
\begin{equation}\label{eq:square-bound}
\sum_{i \in [m]} \lambda_i X_i^2\preceq \Mtwo\id_d.
\end{equation}

Fix \(\delta>0\). For a self-adjoint matrix \(Y\), we define the symmetric exponential potential by
\[
\pot{\delta}{Y}:=\tr e^{\delta Y}+\tr e^{-\delta Y}.
\]
We postpone a more detailed discussion of the origin and the role of this potential to the next subsection. At this stage, we only emphasize that it is well suited to a one-step averaging argument.

We also define
\begin{equation}\label{eq:psi}
\psi_{\Mone}(\delta):=\frac{e^{\delta \Mone}-1-\delta \Mone}{\Mone^2},
\end{equation}
which will appear explicitly as the second-order term of the Taylor expansion
(see also \cite[Chapter 7]{tropp2015introduction}).

The following theorem is the general form of our averaging technique.

\begin{thm}[One-step averaging]\label{thm:onestep}
Fix \(\delta>0\), and let \(Y\) be any self-adjoint matrix. Under assumptions \eqref{eq:mean-zero}, \eqref{eq:norm-bound}, and \eqref{eq:square-bound}, one has
\begin{equation}\label{eq:one-step-average}
\sum_{i \in [m]} \lambda_i \pot{\delta}{Y+X_i}
\le
e^{\Mtwo\psi_{\Mone}(\delta)}\pot{\delta}{Y}.
\end{equation}
Consequently, there exists an index \(i_*\in[m]\) such that
\begin{equation}\label{eq:one-step-choice}
\pot{\delta}{Y+X_{i_*}}
\le
e^{\Mtwo\psi_{\Mone}(\delta)}\pot{\delta}{Y}.
\end{equation}
\end{thm}

Our greedy algorithm for the main problem is to choose, at each step, the next centered matrix, that is, the next matrix of the form \(A_i-\id_d\), so as to minimize the potential. Thus, \Href{Theorem}{thm:onestep} yields a recursive upper bound on the value of the potential at the final step. On the other hand, the potential admits a simple lower bound in terms of the operator norm, which allows us to convert an upper bound for the potential into an upper bound for the approximation error.

\begin{lem}\label{lem:potential-controls}
For every self-adjoint matrix \(Y\),
\begin{equation}\label{eq:potential-controls-norm}
e^{\delta \norm{Y}}\le \pot{\delta}{Y}.
\end{equation}
\end{lem}

\begin{proof}
Since \(Y\) is self-adjoint, either \(\lambda_{\max}(Y)=\norm{Y}\) or \(\lambda_{\min}(Y)=-\norm{Y}\).

If \(\lambda_{\max}(Y)=\norm{Y}\), then
\[
\tr e^{\delta Y}
=
\sum_{j \in [d]} e^{\delta \lambda_j(Y)}
\ge
e^{\delta \lambda_{\max}(Y)}
=
e^{\delta \norm{Y}}.
\]

If \(\lambda_{\min}(Y)=-\norm{Y}\), then
\[
\tr e^{-\delta Y}
=
\sum_{j \in [d]} e^{-\delta \lambda_j(Y)}
\ge
e^{-\delta \lambda_{\min}(Y)}
=
e^{\delta \norm{Y}}.
\]

In either case,
\[
e^{\delta \norm{Y}}
\le
\tr e^{\delta Y}+\tr e^{-\delta Y}
=
\pot{\delta}{Y}.
\]
\end{proof}
\subsection{Motivation for the potential}

We now explain why the potential introduced above is natural for the present problem.

A natural first attempt is to view the equal-weight PSD sparsification problem as a no-dimensional Carath\'eodory-type problem. Recall that such a theorem asks for the following: if \(0\) belongs to  the convex hull of $Q$, where \(Q\) is contained in the unit ball of a Banach space, can one choose \(k\) points of \(Q\) whose average has small norm, with a bound depending on \(k\) but not on the ambient dimension? The classical probabilistic form of this principle is Maurey's lemma \cite{pisier1980remarques}, and in uniformly smooth Banach spaces one also has greedy deterministic versions \cite{ivanov2021approximate}.

In our setting, after centering
\[
X_i:=A_i-\id_d,
\]
we seek a short average of the matrices \(X_i\) under the condition
\[
\sum_{i\in[m]} \lambda_i X_i=0.
\]
This is exactly the form of an approximate Carath\'eodory problem. The difficulty is that the natural norm here is the operator norm, whereas no-dimensional Carath\'eodory theorems are much more effective in uniformly smooth spaces. A standard way around this difficulty is to replace the operator norm by a nearby Schatten norm. Recall that for \(1\le p<\infty\), the Schatten norm \(\norm{B}_{S_p}\) is the \(\ell_p\)-norm of the singular values of \(B\). Thus the spaces \(S_p\) play the role of matrix-valued \(L_p\)-spaces.

However, if one applies this philosophy directly in the present setting, one gets the wrong scale. To see the obstruction, consider the rank-one isotropic case
\[
A_i=d\,u_i\otimes u_i,
\qquad
\sum_{i\in[m]} \lambda_i A_i=\id_d.
\]
Then
\[
X_i=d\,u_i\otimes u_i-\id_d.
\]
The standard substitution trick \cite{ivanov2026nodimensionalresults} is to pass from the operator norm to a nearby Schatten norm and then return to the operator norm at the end. In our situation, one takes
\[
S_\infty \longrightarrow S_p,
\qquad
p=\ln d.
\]
This is also the substitution used by Rudelson \cite{rudelson1999random}. In the rank-one case one has the exact identity
\[
\norm{X_i}=d-1.
\]
Accordingly, the direct no-dimensional Carath\'eodory approach in \(S_p\) leads to a bound of the form
\[
\norm{\frac1k\sum_{r\in[k]} X_{i_r}}_{S_p}
\le
C d\sqrt{\frac{p}{k}},
\]
and therefore also
\[
\norm{\frac1k\sum_{r\in[k]} X_{i_r}}_{S_\infty}
\le
C d\sqrt{\frac{p}{k}}.
\]
So in order to obtain a constant error one needs at least on the order of \(d^2p\) terms. For \(p=\ln d\), this is still quadratic in \(d\), up to the logarithmic factor, and is therefore far from the logarithmic Rudelson scale. Thus the no-dimensional Carath\'eodory viewpoint by itself is not strong enough for the PSD sparsification problem.

So one has to use additional structure. This additional structure is already visible in Rudelson's probabilistic argument. Let \(Q_1,\dots,Q_k\) be independent random matrices taking the values \(A_1,\dots,A_m\) with probabilities \(\lambda_1,\dots,\lambda_m\), and set
\[
D:=\frac1k\sum_{\ell\in[k]} Q_\ell-\id_d.
\]
Choose \(p\ge 2\). Very schematically, by passing from the operator norm to the Schatten \(p\)-norm, then symmetrizing, and finally applying the Lust-Piquard inequality \cite{LuP91}, one arrives at
\begin{align*}
\EE \norm{D}
&\le
\EE \norm{D}_{S_p}\le
\frac{2}{k}\,
\EE_{Q_1,\dots,Q_k}\EE_{r}
\left\|
\sum_{\ell\in[k]} r_\ell Q_\ell
\right\|_{S_p}\\
&\le
\frac{c_0\sqrt{p}}{k}\,
\EE_{Q_1,\dots,Q_k}
\left\|
\left(\sum_{\ell\in[k]} Q_\ell^2\right)^{1/2}
\right\|_{S_p},
\end{align*}
where \(r_1,\dots,r_k\) are independent Rademacher variables.

Up to this point there is essentially no difference between probabilistic proofs of no-dimensional Carath\'eodory theorems and Rudelson's argument. The decisive extra input comes from the cone of positive semidefinite matrices. Since the matrices \(Q_\ell\) are positive semidefinite and satisfy \(\norm{Q_\ell}\le M\), one has
\[
Q_\ell^2\preceq \norm{Q_\ell}\,Q_\ell\preceq M Q_\ell
\qquad\text{for all }\ell\in[k].
\]
Hence
\[
\sum_{\ell\in[k]} Q_\ell^2\preceq M\sum_{\ell\in[k]} Q_\ell,
\]
and therefore
\[
\left(\sum_{\ell\in[k]} Q_\ell^2\right)^{1/2}
\preceq
\sqrt{M}\left(\sum_{\ell\in[k]} Q_\ell\right)^{1/2}.
\]
This is exactly the place where one factors out a square and gains one power of \(M\). From this point on, the rest of the argument is elementary arithmetic, and, after taking \(p=\ln d\), one obtains a logarithmic bound.

This computation indicates what a deterministic proof should preserve. The first condition
\[
\sum_{i\in[m]} \lambda_i X_i=0
\]
is the usual centering condition from the no-dimensional Carath\'eodory approach. But in the PSD problem there is also a second, less visible piece of structure, namely the quadratic estimate
\[
\sum_{i\in[m]} \lambda_i X_i^2\preceq M\id_d,
\]
proved below in \Href{Lemma}{lem:centered}. In other words, the problem is governed not only by the size of the increments \(X_i\), but also by an averaged bound for their squares.

Once this is recognized, the choice of the potential becomes much less mysterious. One wants a one-step inequality for a running error matrix \(Y\) in which the linear term disappears because of the identity
\[
\sum_{i\in[m]} \lambda_i X_i=0,
\]
while the quadratic term is controlled by
\[
\sum_{i\in[m]} \lambda_i X_i^2.
\]
The exponential function is the natural candidate for several related reasons. First, exponentials convert additive perturbations into multiplicative expressions, and such multiplicative estimates are much better suited for iteration. Second, the matrices \(e^{\delta Y}\) and \(e^{\delta X_i}\) are positive definite, so after passing to traces one can use positivity together with the order structure on self-adjoint matrices. Most importantly, the exponential is precisely the device that allows one to cope with noncommutativity: although \(e^{\delta(Y+X_i)}\) does not split into \(e^{\delta Y}e^{\delta X_i}\), a brilliant concavity result of Lieb \cite{lieb1973convex}, \cite[Theorem 8.1.1]{tropp2015introduction} provides us with a straightforward averaging tool:

\begin{prp}[Lieb's concavity theorem]\label{prp:Lieb}
Let \(H\) be a fixed self-adjoint matrix. The map
\[
A\longmapsto \tr \exp\parenth{H+\log A}
\]
is concave on the cone of positive definite matrices \(A\succ 0\).
\end{prp}

\begin{rem}
Alternatively, one can use the Golden--Thompson inequality \cite[Inequality (3.3.3)]{tropp2015introduction}:  
 If \(U\) and \(V\) are self-adjoint matrices, then
\[
\tr e^{U+V}\le \tr\parenth{e^U e^V}.
\]
\end{rem}

This is exactly what makes the trace effective in the noncommutative setting. After applying Lieb's concavity theorem, the one-step estimate reduces to bounding a matrix moment generating function. At that stage the scalar functional
\[
\psi_M(\delta)
=
\frac{e^{\delta M}-1-\delta M}{M^2}
\]
appears explicitly in the inequalities: it is the normalized second-order term in the Taylor expansion of the exponential, and it measures the quadratic remainder after the linear contribution has vanished because of the centering condition.

Finally, since we need to control both positive and negative spectral deviations of \(Y\), it is natural to work with the symmetric potential
\[
\pot{\delta}{Y}:=\tr e^{\delta Y}+\tr e^{-\delta Y}.
\]
The one-step averaging theorem shows that this potential is almost multiplicative under a greedy choice of the next summand. After that, the rest of the proof is a matter of iterating the one-step bound and choosing the parameter \(\delta\) appropriately: in the fixed-\(N\) argument one keeps \(\delta\) constant, whereas in the proof of \Href{Theorem}{thm:all-steps_sparsification} one allows \(\delta\) to vary with the step number.

\section{Proof of the main theorem}
\label{sec:proof-main}

In this section we return to the matrices
\[
A_1,\dots,A_m \succeq 0
\]
from \Href{Theorem}{thm:all-steps_sparsification}, and define the centered matrices
\[
X_i:=A_i-\id_d.
\]
Then
\begin{equation}\label{eq:average-via-sum}
\frac1N\sum_{r \in [N]} A_{i_r}-\id_d
=
\frac1N\sum_{r \in [N]} X_{i_r}.
\end{equation}
Thus it suffices to control averages of the matrices \(X_i\).

We first check that the family \(X_1,\dots,X_m\) satisfies the assumptions of the general averaging theorem with
\[
\Mone=\Mtwo=M.
\]

\begin{lem}\label{lem:centered}
For the matrices \(X_i=A_i-\id_d\), one has
\begin{equation}\label{eq:centered-mean-zero}
\sum_{i \in [m]} \lambda_i X_i=0,
\end{equation}
\begin{equation}\label{eq:Xi-norm}
\norm{X_i}\le M
\qquad\text{for all }i\in[m],
\end{equation}
and
\begin{equation}\label{eq:Xi-square}
\sum_{i \in [m]} \lambda_i X_i^2\preceq M\id_d.
\end{equation}
\end{lem}

\begin{proof}
The identity \eqref{eq:centered-mean-zero} follows immediately from \eqref{eq:original_sum_matrices}:
\[
\sum_{i \in [m]} \lambda_i X_i
=
\sum_{i \in [m]} \lambda_i A_i-\id_d
=
0.
\]

Since \(A_i\succeq 0\) and \(\sum_i \lambda_i A_i=\id_d\), one necessarily has \(M\ge 1\). Also,
\[
-\id_d\preceq X_i=A_i-\id_d\preceq (M-1)\id_d,
\]
and therefore \(\norm{X_i}\le M\), which proves \eqref{eq:Xi-norm}.

Next,
\[
X_i^2=(A_i-\id_d)^2=A_i^2-2A_i+\id_d.
\]
Averaging and using \eqref{eq:original_sum_matrices}, we obtain
\[
\sum_{i \in [m]} \lambda_i X_i^2
=
\sum_{i \in [m]} \lambda_i A_i^2-\id_d
\preceq
\sum_{i \in [m]} \lambda_i A_i^2.
\]
Since \(A_i\succeq 0\) and \(\norm{A_i}\le M\),
\[
A_i^2\preceq \norm{A_i}A_i\preceq M A_i.
\]
Therefore,
\[
\sum_{i \in [m]} \lambda_i A_i^2
\preceq
M\sum_{i \in [m]} \lambda_i A_i
=
M\id_d.
\]
Substituting this into the previous inequality yields
\[
\sum_{i \in [m]} \lambda_i X_i^2\preceq M\id_d.
\]
This proves \eqref{eq:Xi-square}.
\end{proof}

\subsection{Sparsification with prescribed number of steps}

To isolate the main idea, we first prove a fixed-\(N\) version of the theorem. The argument already contains the essential greedy step; the only additional ingredient needed for the all-step statement is to vary the parameter \(\delta\) with the step number.

We also state the elementary bound on \(\psi_M\) that will be used throughout this subsection:
\begin{lem}\label{lem:psi-quadratic}
If \(0\le \delta\le 1/M\), then
\[
\psi_M(\delta)\le \delta^2.
\]
\end{lem}

\begin{proof}
Set \(u:=M\delta\). Then \(0\le u\le 1\), and
\[
\psi_M(\delta)=\frac{e^u-1-u}{M^2}.
\]
For \(0\le u\le 1\), one has
\[
e^u-1-u\le u^2.
\]
Therefore,
\[
\psi_M(\delta)
\le
\frac{u^2}{M^2}
=
\delta^2.
\]
\end{proof}

\begin{thm}\label{thm:main}
Under the assumptions of \Href{Theorem}{thm:all-steps_sparsification}, for every integer \(N\ge 1\), there exists a deterministic sequence of indices
\(
i_1,\dots,i_N\in[m]
\)
such that the following holds.

\begin{enumerate}
\item[(i)] If
\[
N\ge M\ln(2d),
\]
then
\begin{equation}\label{eq:main-largeN}
\norm{
\frac1N\sum_{r \in [N]} A_{i_r}-\id_d
}
\le
2\sqrt{\frac{M\ln(2d)}{N}}.
\end{equation}

\item[(ii)] If
\[
N< M\ln(2d),
\]
then
\begin{equation}\label{eq:main-smallN}
\norm{
\frac1N\sum_{r \in [N]} A_{i_r}-\id_d
}
\le
2\,\frac{M\ln(2d)}{N}.
\end{equation}
\end{enumerate}
\end{thm}

\begin{proof}
By \Href{Lemma}{lem:centered}, the matrices \(X_1,\dots,X_m\) satisfy the assumptions of \Href{Theorem}{thm:onestep} with
\[
\Mone=\Mtwo=M.
\]

Fix \(N\ge 1\) and choose a parameter \(\delta>0\), to be specified later. We construct the sequence \(i_1,\dots,i_N\) recursively.

Set
\[
Y_0:=0.
\]
Suppose \(Y_k\) and \(i_1,\dots,i_k\) have already been chosen. By \Href{Theorem}{thm:onestep}, there exists an index \(i_{k+1}\in[m]\) such that
\[
\pot{\delta}{Y_k+X_{i_{k+1}}}
\le
 e^{M\psi_M(\delta)}\pot{\delta}{Y_k}.
\]
We choose such an index and define
\[
Y_{k+1}:=Y_k+X_{i_{k+1}}.
\]

By construction,
\[
\pot{\delta}{Y_{k+1}}
\le
 e^{M\psi_M(\delta)}\pot{\delta}{Y_k}
\qquad\text{for every }k=0,1,\dots,N-1.
\]
Iterating, we obtain
\[
\pot{\delta}{Y_N}
\le
 e^{NM\psi_M(\delta)}\pot{\delta}{Y_0}.
\]
Since \(Y_0=0\),
\[
\pot{\delta}{Y_0}=\tr \id_d+\tr \id_d=2d.
\]
Therefore,
\begin{equation}\label{eq:potential-final}
\pot{\delta}{Y_N}\le 2d\,e^{NM\psi_M(\delta)}.
\end{equation}

By \Href{Lemma}{lem:potential-controls},
\[
e^{\delta\norm{Y_N}}
\le
\pot{\delta}{Y_N}.
\]
Combining this with \eqref{eq:potential-final}, we get
\[
e^{\delta\norm{Y_N}}
\le
2d\,e^{NM\psi_M(\delta)}.
\]
Taking logarithms gives
\begin{equation}\label{eq:Yn-delta}
\norm{Y_N}
\le
\frac{\ln(2d)}{\delta}
+
\frac{NM\psi_M(\delta)}{\delta}.
\end{equation}

Set
\[
L:=\ln(2d).
\]
Choose
\[
\delta:=\min \braces{\frac1M,\sqrt{\frac{L}{MN}}}.
\]
Since \(\delta\le 1/M\), \Href{Lemma}{lem:psi-quadratic} gives
\[
M\psi_M(\delta)\le M\delta^2.
\]
Therefore \eqref{eq:Yn-delta} yields
\[
\norm{Y_N}
\le
\frac{L}{\delta}+NM\delta.
\]

We distinguish two cases.

\medskip
\noindent
\textbf{Case 1: \(N\ge ML\).}
Then
\[
\delta=\sqrt{\frac{L}{MN}}.
\]
Consequently,
\[
\norm{Y_N}
\le
2\sqrt{MNL}.
\]
Hence
\[
\norm{\frac{1}{N} Y_N}
\le
2\sqrt{\frac{ML}{N}}.
\]
By \eqref{eq:average-via-sum}, this proves \eqref{eq:main-largeN}.

\medskip
\noindent
\textbf{Case 2: \(N<ML\).}

Then
\[
\delta=\frac1M.
\]
Therefore,
\[
\norm{Y_N}
\le
ML+N.
\]
Dividing by \(N\), we obtain
\[
\norm{\frac{1}{N} Y_N}
\le
\frac{ML}{N}+1.
\]
Since in this case \(N<ML\), we have
\[
\frac{ML}{N}>1.
\]
Thus,
\[
\norm{\frac{1}{N} Y_N}
\le
2\,\frac{ML}{N}.
\]
By \eqref{eq:average-via-sum}, this proves \eqref{eq:main-smallN}.
\end{proof}

\subsection{A sequence with control at every step}

In this subsection we prove \Href{Theorem}{thm:all-steps_sparsification}. The argument begins exactly as in the proof above. Up to the threshold \(k\le M\ln(2d)\) we keep the parameter \(\delta\) fixed, and beyond that point we let it decrease with \(k\). The only new ingredient is an interpolation estimate that allows us to compare the potential at two nearby values of \(\delta\):

\begin{lem}\label{lem:delta-interpolation}
Let \(Y\) be a self-adjoint matrix, and let \(0\le \eta\le \delta\). Then
\[
\pot{\eta}{Y}
\le
(2d)^{1-\eta/\delta}\pot{\delta}{Y}^{\eta/\delta}.
\]
\end{lem}

\begin{proof}
Set \(t:=\eta/\delta\in[0,1]\). By H\"older's inequality for the trace,
\[
\tr e^{\eta Y}=\tr\parenth{(e^{\delta Y})^t\id_d^{1-t}}
\le \parenth{\tr e^{\delta Y}}^t \parenth{\tr \id_d}^{1-t}
= d^{1-t}\parenth{\tr e^{\delta Y}}^t.
\]
The same bound holds with \(Y\) replaced by \(-Y\). Summing the two inequalities and using
\[
a^t+b^t\le 2^{1-t}(a+b)^t
\qquad \text{for all } a,b\ge 0,
\]
we obtain
\[
\pot{\eta}{Y}
\le d^{1-t}\parenth{\parenth{\tr e^{\delta Y}}^t+\parenth{\tr e^{-\delta Y}}^t}
\le (2d)^{1-t}\pot{\delta}{Y}^t.
\]
Since \(t=\eta/\delta\), this is exactly the claimed estimate.
\end{proof}

\begin{proof}[Proof of \Href{Theorem}{thm:all-steps_sparsification}]
Let
\[
X_i:=A_i-\id_d.
\]
By \Href{Lemma}{lem:centered},
\[
\sum_{i \in [m]} \lambda_i X_i=0,
\qquad
\norm{X_i}\le M,
\qquad
\sum_{i \in [m]} \lambda_i X_i^2\preceq M\id_d.
\]

Set
\[
L:=\ln(2d),
\qquad
q:=\floor{ML}+1.
\]
Thus, for integers \(k\ge 1\),
\[
k<q \iff k\le ML,
\qquad
k\ge q \iff k>ML.
\]
For every \(k\ge 1\), define
\[
\delta_k:=\min\braces{\frac1M,\sqrt{\frac{L}{Mk}}}.
\]
Equivalently,
\[
\delta_k=\frac1M
\qquad\text{for }k<q,
\]
and
\[
\delta_k=\sqrt{\frac{L}{Mk}}
\qquad\text{for }k\ge q.
\]

We construct the sequence recursively. Set
\[
Y_0:=0.
\]
Suppose \(Y_{k-1}\) and \(i_1,\dots,i_{k-1}\) have already been chosen. By \Href{Theorem}{thm:onestep}, applied with the parameter \(\delta_k\), there exists an index \(i_k\in[m]\) such that
\[
\pot{\delta_k}{Y_{k-1}+X_{i_k}}
\le
 e^{a_k}\pot{\delta_k}{Y_{k-1}},
\qquad
a_k:=M\psi_M(\delta_k).
\]
We choose such an index and define
\[
Y_k:=Y_{k-1}+X_{i_k}.
\]
Then
\begin{equation}\label{eq:stepwise-potential}
\pot{\delta_k}{Y_k}
\le
 e^{a_k}\pot{\delta_k}{Y_{k-1}}
\qquad\text{for all }k\ge 1.
\end{equation}

We now estimate the partial sums \(Y_k\).

\medskip
\noindent
\textbf{Step 1: the range \(k<q\), equivalently \(k\le ML\).}

In this range the parameter is constant: \(\delta_j=1/M\) for every \(j\le k\). Thus the argument is exactly the same as in the proof of \Href{Theorem}{thm:main}. Since \(\delta_j=1/M\), \Href{Lemma}{lem:psi-quadratic} gives
\[
a_j=M\psi_M(1/M)\le \frac1M
\qquad\text{for every }j\le k.
\]
Iterating \eqref{eq:stepwise-potential}, we obtain
\[
\pot{1/M}{Y_k}
\le
2d\,e^{k/M}.
\]
By \Href{Lemma}{lem:potential-controls},
\[
e^{\norm{Y_k}/M}\le \pot{1/M}{Y_k},
\]
and therefore
\[
\norm{Y_k}
\le
ML+k.
\]
Dividing by \(k\), we get
\[
\norm{\frac1k Y_k}
\le
\frac{ML}{k}+1.
\]
Since \(k\le ML\), one has \(ML/k\ge 1\), and hence
\[
\norm{\frac1k Y_k}
\le
2\,\frac{ML}{k}.
\]
Using \eqref{eq:average-via-sum}, we conclude that
\[
\norm{
\frac1k\sum_{r \in [k]} A_{i_r}-\id_d
}
\le
2\,\frac{M\ln(2d)}{k}
\qquad\text{for }k<q.
\]

\medskip
\noindent
\textbf{Step 2: the range \(k\ge q\), equivalently \(k>ML\).}

Set
\[
B_k:=\pot{\delta_k}{Y_k},
\qquad
b_k:=\ln B_k,
\qquad
c_k:=b_k-L.
\]

We first obtain the initial bound
\begin{equation}\label{eq:cq-bound}
c_q\le 2L.
\end{equation}
If \(q=1\), then \(ML<1\), so \(\delta_1=\sqrt{L/M}\le 1/M\). Hence \Href{Lemma}{lem:psi-quadratic} gives
\[
a_1=M\psi_M(\delta_1)
\le
M\delta_1^2
=
L.
\]
Since \(Y_0=0\), \eqref{eq:stepwise-potential} gives
\[
B_1
\le
 e^{a_1}\pot{\delta_1}{0}
=
2d\,e^{a_1},
\]
so \(c_1\le L\), which proves \eqref{eq:cq-bound} in this case.

Assume now that \(q\ge 2\). So, Step 1 applies at time \(q-1\). In particular,
\[
\pot{1/M}{Y_{q-1}}
\le
2d\,e^{(q-1)/M},
\]
and therefore
\[
c_{q-1}=\ln \pot{1/M}{Y_{q-1}}-L
\le
\frac{q-1}{M}
\le
L.
\]
Since \(q>ML\), we have \(\delta_q=\sqrt{L/(Mq)}\le 1/M\). Applying \Href{Lemma}{lem:delta-interpolation} with
\[
\alpha_q:=\frac{\delta_q}{\delta_{q-1}}=M\delta_q\in(0,1],
\]
and then using \eqref{eq:stepwise-potential}, we obtain
\[
B_q
\le
 e^{a_q}(2d)^{1-\alpha_q}B_{q-1}^{\alpha_q}.
\]
Taking logarithms gives
\[
c_q\le a_q+\alpha_q c_{q-1}.
\]
Since \(\delta_q\le 1/M\), \Href{Lemma}{lem:psi-quadratic} gives
\[
a_q=M\psi_M(\delta_q)
\le
M\delta_q^2
=
\frac{L}{q}
\le L.
\]
Since \(c_{q-1}\le L\), we obtain \(c_q\le 2L\), proving \eqref{eq:cq-bound}.

We next claim that
\begin{equation}\label{eq:ck-bound}
c_k\le 2L
\qquad\text{for all }k\ge q.
\end{equation}
We have already established the base case \(k=q\). Now let \(k\ge q+1\), and assume that \(c_{k-1}\le 2L\). Since both \(k-1\) and \(k\) belong to the second range,
\[
\delta_{k-1}=\sqrt{\frac{L}{M(k-1)}},
\qquad
\delta_k=\sqrt{\frac{L}{Mk}},
\qquad
\alpha_k:=\frac{\delta_k}{\delta_{k-1}}=\sqrt{\frac{k-1}{k}}.
\]
Applying \Href{Lemma}{lem:delta-interpolation} and then \eqref{eq:stepwise-potential}, we get
\[
B_k
\le
 e^{a_k}(2d)^{1-\alpha_k}B_{k-1}^{\alpha_k},
\]
so
\[
c_k\le a_k+\alpha_k c_{k-1}.
\]
Again \(\delta_k\le 1/M\), and therefore \Href{Lemma}{lem:psi-quadratic} gives
\[
a_k=M\psi_M(\delta_k)
\le
M\delta_k^2
=
\frac{L}{k}.
\]
Using the inductive hypothesis, we obtain
\[
c_k
\le
\frac{L}{k}+2L\sqrt{\frac{k-1}{k}}.
\]
Since
\[
\frac1k+2\sqrt{\frac{k-1}{k}}\le 2,
\]
it follows that \(c_k\le 2L\). This proves \eqref{eq:ck-bound}.

Thus, for every \(k\ge q\),
\[
b_k=L+c_k\le 3L.
\]
Using \Href{Lemma}{lem:potential-controls}, we obtain
\[
e^{\delta_k\norm{Y_k}}\le \pot{\delta_k}{Y_k}=B_k,
\]
hence
\[
\norm{Y_k}\le \frac{b_k}{\delta_k}\le \frac{3L}{\delta_k}.
\]
Since \(k\ge q\), we have \(\delta_k=\sqrt{L/(Mk)}\), and therefore
\[
\norm{Y_k}\le 3\sqrt{MLk}.
\]
Dividing by \(k\), we get
\[
\norm{\frac1k Y_k}\le 3\sqrt{\frac{ML}{k}}.
\]
Using again \eqref{eq:average-via-sum}, we arrive at
\[
\norm{
\frac1k\sum_{r \in [k]} A_{i_r}-\id_d
}
\le
3\sqrt{\frac{M\ln(2d)}{k}}
\qquad\text{for }k\ge q.
\]

Combining the two regimes proves the theorem.
\end{proof}

\section{Proof of the one-step averaging theorem}
\label{sec:ave_thm}

In this section we prove \Href{Theorem}{thm:onestep}. We begin by collecting the auxiliary estimates that enter the argument.

\subsection{Auxiliary estimates}

We will use two standard matrix inequalities.

\begin{prp}[Trace comparison]\label{prp:trace-comparison}
Let \(A,B\succeq 0\). If
\[
B\preceq c\,\id_d
\]
for some \(c\ge 0\), then
\[
\tr(AB)\le c\,\tr(A).
\]
\end{prp}

\begin{prp}\label{prp:spectral-transfer}
Let \(S=S^*\) be a self-adjoint matrix, and let \(f,g:\R\to\R\) be analytic functions such that
\[
f(x)\le g(x)
\qquad\text{for every }x\in \sigma(S),
\]
where \(\sigma(S)\) denotes the spectrum of \(S\). Then
\[
f(S)\preceq g(S).
\]
\end{prp}

Both statements are standard, so we omit their proofs.

The next lemma isolates the scalar inequality behind the potential. Combined with \Href{Proposition}{prp:spectral-transfer}, it allows us to compare the matrix exponential with a quadratic polynomial in the same matrix. This is exactly the step that makes it possible to use \eqref{eq:mean-zero} to cancel the linear term and \eqref{eq:square-bound} to control the quadratic term.

\begin{lem}\label{lem:scalar}
For every real number \(x\le \Mone\),
\begin{equation}\label{eq:scalar-bound}
e^{\delta x}\le 1+\delta x+\psi_{\Mone}(\delta)x^2.
\end{equation}
\end{lem}

\begin{proof}
Define
\[
f(x):=
\begin{cases}
\dfrac{e^{\delta x}-1-\delta x}{x^2}, & x\neq 0,\\[1.5ex]
\dfrac{\delta^2}{2}, & x=0.
\end{cases}
\]
Using the identity
\[
e^u-1-u=u^2\int_0^1 (1-s)e^{su}\,\di s,
\]
we obtain
\[
f(x)=\delta^2\int_0^1 (1-s)e^{\delta s x}\,\di s.
\]
This representation shows that \(f\) is increasing in \(x\). Therefore, for every \(x\le \Mone\),
\[
f(x)\le f(\Mone)=\psi_{\Mone}(\delta).
\]
Multiplying by \(x^2\), we get
\[
e^{\delta x}-1-\delta x\le \psi_{\Mone}(\delta)x^2,
\]
which is exactly \eqref{eq:scalar-bound}.
\end{proof}

We now transfer this estimate to matrices.

\subsection{Proof of the one-step theorem}

We first state the matrix exponential bounds that will be used in the proof.

\begin{lem}\label{lem:mgf}
Under assumptions \eqref{eq:mean-zero}, \eqref{eq:norm-bound}, and \eqref{eq:square-bound}, one has
\begin{equation}\label{eq:mgf-plus}
\sum_{i \in [m]} \lambda_i e^{\delta X_i}\preceq e^{\Mtwo\psi_{\Mone}(\delta)}\id_d,
\end{equation}
and
\begin{equation}\label{eq:mgf-minus}
\sum_{i \in [m]} \lambda_i e^{-\delta X_i}\preceq e^{\Mtwo\psi_{\Mone}(\delta)}\id_d.
\end{equation}
\end{lem}

\begin{proof}
Fix \(i\in[m]\). Since \(X_i\) is self-adjoint and \(\norm{X_i}\le \Mone\), every eigenvalue \(\mu\) of \(X_i\) satisfies
\[
\mu\le \Mone.
\]
Applying \Href{Proposition}{prp:spectral-transfer} to \(X_i\) and using \Href{Lemma}{lem:scalar}, we obtain
\[
e^{\delta X_i}\preceq \id_d+\delta X_i+\psi_{\Mone}(\delta)X_i^2.
\]
Multiplying by \(\lambda_i\) and summing over \(i\), we get
\[
\sum_{i \in [m]} \lambda_i e^{\delta X_i}
\preceq
\sum_{i \in [m]} \lambda_i \id_d
+
\delta\sum_{i \in [m]} \lambda_i X_i
+
\psi_{\Mone}(\delta)\sum_{i \in [m]} \lambda_i X_i^2.
\]
Using \eqref{eq:mean-zero}, \(\sum_i \lambda_i=1\), and \eqref{eq:square-bound}, we obtain
\[
\sum_{i \in [m]} \lambda_i e^{\delta X_i}
\preceq
\parenth{1+\Mtwo\psi_{\Mone}(\delta)}\id_d
\preceq
e^{\Mtwo\psi_{\Mone}(\delta)}\id_d.
\]
This proves \eqref{eq:mgf-plus}.

Applying \eqref{eq:mgf-plus} to the family \(-X_1,\dots,-X_m\), which satisfies the same assumptions, yields
\[
\sum_{i \in [m]} \lambda_i e^{-\delta X_i}
=
\sum_{i \in [m]} \lambda_i e^{\delta(-X_i)}
\preceq
e^{\Mtwo\psi_{\Mone}(\delta)}\id_d.
\]
This proves \eqref{eq:mgf-minus}.
\end{proof}

We now combine the matrix exponential bounds from \Href{Lemma}{lem:mgf} with Lieb's concavity theorem.

\begin{proof}[Proof of \Href{Theorem}{thm:onestep}]

We first treat the positive exponential part.  Applying \Href{Proposition}{prp:Lieb} with
\(H=\delta Y\) and \(A_i=e^{\delta X_i}\), we obtain
\[
\sum_{i\in[m]}\lambda_i\tr e^{\delta Y+\delta X_i}
=
\sum_{i\in[m]}\lambda_i\tr\exp\parenth{\delta Y+
\log e^{\delta X_i}}
\le
\tr\exp\parenth{\delta Y+
\log\parenth{\sum_{i\in[m]}\lambda_i e^{\delta X_i}}}.
\]
By \eqref{eq:mgf-plus},
\[
\sum_{i\in[m]}\lambda_i e^{\delta X_i}
\preceq
 e^{\Mtwo\psi_{\Mone}(\delta)}\id_d.
\]
Using the operator monotonicity of the logarithm and the monotonicity of
\(A\mapsto\tr e^A\) on self-adjoint matrices, we get
\[
\tr\exp\parenth{\delta Y+
\log\parenth{\sum_{i\in[m]}\lambda_i e^{\delta X_i}}}
\le
\tr\exp\parenth{\delta Y+
\Mtwo\psi_{\Mone}(\delta)\id_d}
=
 e^{\Mtwo\psi_{\Mone}(\delta)}\tr e^{\delta Y}.
\]
Thus
\begin{equation}\label{eq:one-step-plus-lieb}
\sum_{i\in[m]}\lambda_i\tr e^{\delta(Y+X_i)}
\le
 e^{\Mtwo\psi_{\Mone}(\delta)}\tr e^{\delta Y}.
\end{equation}

The negative exponential part is identical.  Applying Lieb's theorem with
\(H=-\delta Y\) and \(A_i=e^{-\delta X_i}\), and then using
\eqref{eq:mgf-minus}, gives
\begin{equation}\label{eq:one-step-minus-lieb}
\sum_{i\in[m]}\lambda_i\tr e^{-\delta(Y+X_i)}
\le
 e^{\Mtwo\psi_{\Mone}(\delta)}\tr e^{-\delta Y}.
\end{equation}
Adding \eqref{eq:one-step-plus-lieb} and \eqref{eq:one-step-minus-lieb}, we obtain
\[
\sum_{i\in[m]}\lambda_i\pot{\delta}{Y+X_i}
\le
 e^{\Mtwo\psi_{\Mone}(\delta)}\pot{\delta}{Y},
\]
which is \eqref{eq:one-step-average}.  The existence of an index satisfying
\eqref{eq:one-step-choice} follows again because the left-hand side is a
convex combination of the values \(\pot{\delta}{Y+X_i}\).
\end{proof}

\bibliographystyle{alpha}
\bibliography{../work_current/uvolit}

\end{document}